\documentclass[10pt,a4paper,reqno]{amsart}
\pdfoutput=1

\usepackage[maxbibnames=10,maxcitenames=5]{biblatex}
\defbibheading{bibliography}[\refname]{%
	\section*{#1}%
}

\usepackage{lmodern}
\usepackage[T1]{fontenc}

\usepackage{microtype}

\usepackage{hyperref}
\newcommand{\mailto}[1]{\href{mailto:#1}{#1}}

\usepackage{mathtools}
\usepackage{thmtools}

\declaretheorem[
	style=plain,
	name=Theorem,
	numberwithin=section,
	refname={Theorem,Theorems},
	Refname={Theorem,Theorems}
]{Thm}
\declaretheorem[
	style=plain,
	name=Proposition,
	numberlike=Thm,
	refname={Proposition,Propositions},
	Refname={Proposition,Propositions}
]{Prop}
\declaretheorem[
	style=plain,
	name=Corollary,
	numberlike=Thm,
	refname={Corollary,Corollaries},
	Refname={Corollary,Corollaries}
]{Cor}
\declaretheorem[
	style=plain,
	name=Conjecture,
	numberlike=Thm,
	refname={Conjecture,Conjectures},
	Refname={Conjecture,Conjectures}
]{Conj}
\declaretheorem[
style=plain,
name=Lemma,
numberlike=Thm,
refname={Lemma,Lemmas},
Refname={Lemma,Lemmas}
]{Lem}
\declaretheorem[
style=definition,
name=Definition,
numberlike=Thm,
refname={Definition,Definitions},
Refname={Definition,Definitions}
]{Def}
\declaretheorem[
style=definition,
name=Notation,
numberlike=Thm,
refname={Notation,Notations},
Refname={Notation,Notations}
]{Not}

\newcommand{\wt}{\mathrm{wt}}
\newcommand{\mot}{\mathfrak{m}}
\newcommand{\lmot}{\mathfrak{L}}

\newcommand{\R}{\mathbb{R}}
\newcommand{\Q}{\mathbb{Q}}

\DeclareMathOperator*{\per}{per}
\DeclareMathOperator*{\id}{id}

\newcommand{\isom}{\cong}
\renewcommand{\vec}[1]{{\underline{\boldsymbol{#1}}}}

\usepackage{etoolbox}
\makeatletter
\expandafter\patchcmd\csname\string\proof\endcsname{\itshape}{\bfseries}{}{}

\patchcmd{\@settitle}{\bfseries}{}{}{}
\makeatother

\mathcode`l="8000
\begingroup
\makeatletter
\lccode`\~=`\l
\DeclareMathSymbol{\lsb@l}{\mathalpha}{letters}{`l}
\lowercase{\gdef~{\ifnum\the\mathgroup=\m@ne \ell \else \lsb@l \fi}}%
\endgroup

\begin{document}

	\hypersetup{
		pdfinfo = {
			Title = {zeta(\{\{2\}\textasciicircum{}m, 1, \{2\}\textasciicircum{}m, 3\}\textasciicircum{}n, \{2\}\textasciicircum{}m) / pi\textasciicircum{}(4n + 2m(2n+1)) is rational},
			Author = {Steven Charlton}
		}
	}

	\title{\( \zeta(\{ \, \{2\}^m, 1, \{2\}^m, 3 \}^n, \{2\}^m) / \pi^{4n + 2m(2n+1)} \) is rational}
	\date{\today}

	\author{Steven Charlton}
	\address{Departement of Mathematical Sciences \\
		Durham University \\
		Science Laboratories \\
		South Road \\
		Durham \\
		DH1 3LE \\
		United Kingdom}

	\email{\mailto{steven.charlton@durham.ac.uk}}
	
	\begin{abstract}
		The cyclic insertion conjecture of Borwein, Bradley, Broadhurst and Lison\v{e}k states that inserting all cyclic shifts of some fixed blocks of 2's into the multiple zeta value \( \zeta(1,3,\ldots,1,3) \) gives an explicit rational multiple of a power of \( \pi \).  In this paper we use motivic multiple zeta values to establish a non-explicit symmetric insertion result: inserting all possible permutations of some fixed blocks of 2's into \( \zeta(1,3,\ldots,1,3) \) gives some rational multiple of a power of \( \pi \).
	\end{abstract}
	
	\thanks{I would like to thank Herbert Gangl for many useful comments and suggestions which have improved the exposition.}
	
	\maketitle
	
	\section{Introduction}
	
		In Equation 18 of \cite[][p. 4]{EvalEZ}, \citeauthor{EvalEZ} give a conjectural evaluation of a two parameter family of multiple zeta values:
		\begin{equation}\label{eqn:BBBEval}
			\zeta(\{ \, \{2\}^m, 1, \{2\}^m, 3\}^n, \{2\}^m) \overset{?}{=} \frac{1}{2n+1} \frac{\pi^\wt}{(\wt+1)!}
		\, , \end{equation}
		where I write `\( \wt \)' as shorthand for the weight of the multiple zeta value, which here is equal to \( 4n + 2m(2n+1) \).
		
		Throughout this paper we will abbreviate multiple zeta value to MZV, and keep with the convention that means \( \zeta(1,2) \) is a convergent MZV.  We will make use of the notation
		\[
			\{ s_1, s_2, \ldots, s_k \}^l \coloneqq \underbrace{s_1, s_2, \ldots, s_k, \, \ldots , \, s_1, s_2, \ldots, s_k}_\text{{\( l \) copies of \( s_1, s_2, \ldots, s_k \)}}
		\, , \]
		to write repeated arguments.  Also, for integers \( b_0, b_1, \ldots, b_{2n+1} \geq 0 \), set
		\[
			Z(b_0, b_1, \ldots, b_{2n}) \coloneqq \zeta\big(\{2\}^{b_0}, 1, \{2\}^{b_1}, 3, \ldots, \{2\}^{b_{2n-2}}, 1, \{2\}^{b_{2n-1}}, 3, \{2\}^{b_{2n}}\big)
		\, , \]
		which is obtained by inserting \( \{2\}^{b_i} \) after the \( i \)-th term of \( \{1, 3\}^n \).
		
		Throughout \cite{CombAspectsMZV}, \citeauthor{CombAspectsMZV} present numerical evidence for a cyclic insertion conjecture, Conjecture 1 in \cite[][p. 9]{CombAspectsMZV}, which generalises the above family.  Their conjecture can be given as follows:
		\begin{Conj}[Cyclic Insertion]
			For given integers \( a_0, a_1, \ldots, a_{2n} \geq 0 \)
			\[
				\sum_{\mathclap{r \in C_{2n+1}}} \, Z(a_{r(0)}, a_{r(1)}, \ldots, a_{r(2n)}) \overset{?}{=} \frac{\pi^\wt}{(\wt+1)!}
			\, , \]
			where \( C_{2n+1} \) is the cyclic group of order \( 2n+1 \), acting naturally by cyclically shifting the indices \( 0, 1, \ldots, 2n \) of the \( a_i \)'s.  So all cyclic shifts of the fixed blocks \( \{2\}^{a_0} \), \( \{2\}^{a_1} \), \ldots, \( \{2\}^{a_{2n}} \) are inserted into \( \{1,3\}^n \).
		\end{Conj}
		
		In \cite{AlgCombMZV}, \citeauthor{AlgCombMZV} succeed in proving:
		\begin{Thm}[Bowman-Bradley, {\cite[Theorem 5.1 in][p. 19]{AlgCombMZV}}]
			For given integers \( n, m \geq 0 \)
			\[
				\sum_{\mathclap{\substack{j_0 + j_1 + \cdots + j_{2n} = m \\ j_0, j_1, \ldots, j_{2n} \geq 0}}} \, Z(j_0, j_1, \ldots, j_{2n}) = \frac{1}{2n+1} \binom{m+2n}{m} \frac{\pi^\wt}{(\wt + 1)!}
			\, . \]
			So all blocks \( \{2\}^{j_0} \), \( \{2\}^{j_1} \), \ldots, \( \{2\}^{j_{2n}} \) corresponding to compositions\footnote{Strictly speaking these are \emph{weak} compositions since some of the terms may be 0, but for ease of use I will just call them compositions.} \( \sum_{k=0}^{2n} j_k = m \) of \( m \) into \( 2n+1 \) parts are inserted into \( \{1,3\}^n \).
		\end{Thm}
		Simpler and more refined proofs of this result have since been given by \citeauthor{ZhaoExoticShuffle} \cite{ZhaoExoticShuffle} and \citeauthor{MunetaEvalMZV} \cite{MunetaEvalMZV}.
		
		This result is compatible with the cyclic insertion conjecture.  Any composition \( \sum_{k=0}^{2n} j_k = m \) of \( m \) into \( 2n+1 \) parts remains a composition of \( m \) into \( 2n+1 \) parts when cyclically shifted.  Hence the terms in the Bowman-Bradley sum can be re-grouped into subsums, where each subsum is taken over a set of compositions which differ by a cyclic shift.  Conjecturally, each of these subsums is then a rational multiple of \( \pi^\wt \); explicitly it should be \( \frac{\alpha}{2n+1} \frac{\pi^\wt}{(\wt+1)!} \), where \( \alpha \) is the number of distinct compositions obtained by cyclically shifting a representative composition appearing in this subsum.  So on average each of the \( \binom{m+2n}{m} \) compositions contributes \( \frac{1}{2n+1} \frac{\pi^\wt}{(\wt+1)!} \), giving a total which agrees with the above. \medskip
		
		In this paper we will use Brown's motivic MZV framework \cite{BrownMTM, DecompMot} to prove the non-explicit version of a `symmetric insertion' result:
		\begin{Prop}[Symmetric Insertion]
			For given integers \( a_0, a_1, \ldots, a_{2n} \geq 0 \)
			\[
				\sum_{\mathclap{\sigma \in S_{2n+1}}} \, Z(a_{\sigma(0)}, a_{\sigma(1)}, \ldots, a_{\sigma(2n)}) \in \pi^\wt \Q
			\, , \]
			where \( S_{2n+1} \) is the symmetric group on the \( 2n+1 \) letters \( 0, 1, \ldots, 2n \).  So all possible permutations of the fixed blocks \( \{2\}^{a_i} \) are inserted into \( \{ 1, 3 \}^n \).
		\end{Prop}
		This result sits at an intermediate level between the cyclic insertion conjecture and the Bowman-Bradley theorem.
		
		Any permutation of a composition of \( m \) into \( 2n+1 \) parts remains a composition of \( m \) into \( 2n+1 \) parts, so the Bowman-Bradley sum breaks up into subsums, each over the compositions which differ by a permutation.  By symmetric insertion each of these subsums \emph{is} a rational multiple of \( \pi^\wt \).  
		
		On the other hand, by choosing representatives of the cosets of \( S_{2n+1} / \langle (0 \, 1 \cdots 2n) \rangle \), the sum over \( S_{2n+1} \) breaks up into \( (2n)! \) sums over \( C_{2n+1} \isom \langle (0 \, 1 \cdots 2n) \rangle \).  By the cyclic insertion conjecture, each of these subsums is equal to \( \frac{\pi^\wt}{(\wt+1)!} \), giving the total as \((2n)! \frac{\pi^\wt}{(\wt+1)!} \in \pi^\wt \Q\), so this result is compatible with cyclic insertion. \medskip
	
		As a corollary to symmetric insertion, by setting \( a_0 = a_1 = \cdots = a_{2n} = m \), it will follow that
		\[
			\zeta(\{ \, \{2\}^m, 1, \{ 2^m \}, 3 \}^n, \{2\}^m) \in \pi^\wt \Q
		\, , \]
		that is, a `weak version' of the conjectural evaluation in \autoref{eqn:BBBEval} holds.\medskip
		
		\paragraph{\bfseries{Acknowledgements.}} This work began to take shape thanks to Brown's and Gangl's Multiple Zeta Values lecture series during the Grothendieck-Teichm\"uller Groups, Deformation and Operads programme at the Isaac Newton Institute.  I am grateful to these lecturers and to the organisers of the GDO programme.  I am also grateful to the INI for providing financial support covering the cost of travel to the lectures.  This work was done with the support of Durham Doctoral Scholarship funding.
		
	\section{Motivic Multiple Zeta Values}
	
		In Section 2 of \cite{GaloisSymGon}, Goncharov shows how the classical iterated integrals
		\[ I(a_0; a_1, \ldots, a_n; a_{n+1}) \]
		can be lifted to motivic iterated integrals 
		\[
			I^\mot(a_0; a_1, \ldots, a_n; a_{n+1}) 
		\, , \]
		with new algebraic structure.  This structure comes in the form of a coproduct \( \Delta \), explicitly computed in Theorem 1.2 of \cite[][p. 3]{GaloisSymGon}, making the motivic iterated integrals into a Hopf algebra.
	
		In Section 2 of \cite{BrownMTM}, Brown further lifts Goncharov's motivic iterated integrals, in such a way that \( I^\mot(0; 1, 0; 1) \) and the corresponding motivic MZV \( \zeta^\mot(2) \) are non-zero.  More generally Definition 3.6 of \cite[][p. 8]{DecompMot} defines a motivic MZV as
		\[
			\zeta^\mot(n_1, n_2, \ldots, n_r) \coloneqq (-1)^r I^\mot(0; \underbrace{1, 0, \ldots, 0}_{\text{\( n_1 \) terms}}, \, \underbrace{1, 0, \ldots, 0}_{\text{\( n_2 \) terms}}, \, \ldots \, , \underbrace{1, 0, \ldots, 0}_{\text{\( n_r \) terms}}; 1)
		\, , \]
		in analogy with the Kontsevich integral representation of an MZV, Section 9 in \cite{MZVApp}.  
		
		Brown's motivic MZVs form a graded coalgebra, denoted \( \mathcal{H} \).  The period map
		\begin{equation}\label{eqn:permap}
			\begin{aligned}
				\per \colon \mathcal{H} & \to \R \\
			 	I^\mot(a_0; a_1, \ldots, a_n; a_{n+1}) &\mapsto I(a_0; a_1, \ldots, a_n; a_{n+1})
			\end{aligned}
		\end{equation}
		defines a ring homomorphism from the graded coalgebra \( \mathcal{H} \) to \( \R \), see Equation 2.11 in \cite[][p. 4]{BrownMTM} and Equation 3.8 in \cite[][p. 7]{DecompMot}.  This means any identities between motivic MZVs descend to the same identities between ordinary MZVs.\medskip
		
		Theorem 2.4 of \cite[][p. 6]{BrownMTM} shows that Goncharov's coproduct lifts to a coaction \( \Delta \colon \mathcal{H} \to \mathcal{A} \otimes_\Q\mathcal{H} \) on Brown's motivic MZVs, where \( \mathcal{A} \coloneqq \mathcal{H} / \zeta^\mot(2)\mathcal{H} \) kills \( \zeta^\mot(2) \).  In Section 5 of \cite{DecompMot}, Brown describes an algorithm for decomposing motivic MZVs into a chosen basis using an infinitesimal version of this coaction \( \Delta \colon \mathcal{H} \to \mathcal{A} \otimes_\Q \mathcal{H} \).
		
		The infinitesimal coaction factors through the operators
		\[
			D_r \colon \mathcal{H}_N \to \mathcal{L}_r \otimes_\Q \mathcal{H}_{N-r}
		\, , \]
		where \( \mathcal{L}_r \) is the degree \( r \) component of \( \mathcal{L} \coloneqq \mathcal{A}_{>0} / \mathcal{A}_{>0}\mathcal{A}_{>0} \), the Lie coalgebra of indecomposables, and \( \mathcal{H}_{N} \) is the degree \( N \) component of \( \mathcal{H} \).  The action of \( D_r \) on the motivic iterated integral \( I^\mot(a_0; a_1, \ldots, a_n; a_{n+1}) \) is given explicitly by
		\[
			\sum_{p=0}^{n-r} I^\lmot(a_p; a_{p+1}, \ldots, a_{p+r}; a_{p+r+1}) \otimes I^\mot(a_0; a_1, \ldots, a_p, a_{p+r+1}, \ldots, a_n; a_{n+1})
		\, , \]
		according to Equation 3.4 of \cite[][p. 8]{BrownMTM}.
		
		The operators \( D_r \) have a pictorial interpretation similar to that of Goncharov's coproduct and the coaction above.  One can view \( D_r \) as cutting segments of length \( r \) out of a semicircular polygon whose vertices are decorated by \( a_0, a_1, \ldots, a_n, a_{n+1} \):
		\begin{center}
			\vskip1em
			\includegraphics[scale=0.75]{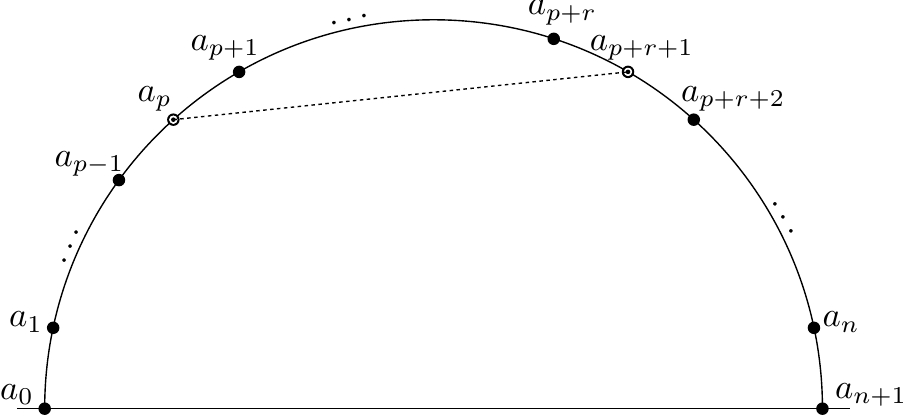}
			\vskip1em
		\end{center}
		Notice that the \emph{boundary terms} \( a_p \) and \( a_{p+r+1} \) appear in both the left and right hand factors of \( D_r \), they are part of both the main polygon and the cut-off segment above.
		
		One could also see the operators \( D_r \) as cutting out strings of length \( r \) from the sequence \( (a_0; a_1, \ldots, a_n; a_{n+1}) \).  Following Brown, Definition 4.4 in \cite[][p. 11]{DecompMot}, we call sequence
		\[
			(a_p; a_{p+1}, \ldots, a_{p+r}; a_{p+r+1})
		\]
		appearing in the left factor of \( D_r \) the \emph{subsequence}, and we call the sequence
		\[
			(a_0; a_1, \ldots, a_p, a_{p+r+1}, \ldots, a_n; a_{n+1})
		\]
		appearing in the right factor of \( D_r \) the \emph{quotient sequence} of the original sequence.  Again the boundary terms \( a_p \) and \( a_{p+r+1} \) are part of both the subsequence and the quotient sequence. \medskip
		
		When decomposing a motivic MZV into a basis, the operator \( D_{2k+1} \) is used to extract the coefficient of \( \zeta^\mot(2k+1) \) as a polynomial in this basis, see Section 5 of \cite{DecompMot}.  The upshot of this comes from Theorem 3.3 of \cite[][p. 9]{BrownMTM}: 
		\begin{Thm}\label{thm:kerDN}
			The kernel of \( D_{<N} \coloneqq \bigoplus_{3 \leq 2k + 1 < N} D_{2k+1} \) is \( \zeta^\mot(N) \Q \) in weight \( N \).
		\end{Thm}
		In other words, if the operators \( D_{2k+1} \), for \( k \) such that \( 3 \leq 2k+1 < N \), all vanish on a given combination of motivic MZVs of weight \( N \), then this combination is a rational multiple of \( \zeta^\mot(N) \).  This will be the main tool in our proof of symmetric insertion.  \medskip
		
		Before we continue we need to recall a few properties of motivic iterated integrals which will be used in the proof, see Section 2.4 of \cite{BrownMTM} for a complete list of properties.  We need:
		\begin{itemize}
			\item \( I^\mot(a_0; a_1, \ldots, a_n; a_{n+1}) = 0 \)  if \( n \geq 1 \) and \( a_0 = a_{n+1} \), and
			\item \( I^\mot(0; a_1, \ldots, a_n; 1) = (-1)^n I^\mot(1; a_n, \ldots, a_1; 0) \).
		\end{itemize}
		We will refer to these properties as the \emph{vanishing because the boundaries are equal}, and \emph{reversal of paths} respectively.
		
		\section{Symmetric Insertion}
		
			\begin{Prop}[Symmetric Insertion]\label{prop:symins}
				For given integers \( a_0, a_1, \ldots, a_{2n} \geq 0 \)
				\[
					\sum_{\mathclap{\sigma \in S_{2n+1}}} \, Z(a_{\sigma(0)}, a_{\sigma(1)}, \ldots, a_{\sigma(2n)}) \in \pi^\wt \Q
				\, , \]
				where \( S_{2n+1} \) is the symmetric group on the \( 2n+1 \) letters \( 0, 1, \ldots, 2n \).  So all possible permutations of the fixed blocks \( \{2\}^{a_i} \) are inserted into \( \{ 1, 3 \}^n \).
			\end{Prop}
			
			\begin{proof}[Strategy of Proof]\let\qed\relax
				We will put
				\[
					S \coloneqq \sum_{\mathclap{\sigma \in S_{2n+1}}} \, Z(a_{\sigma(0)}, a_{\sigma(1)}, \ldots, a_{\sigma(2n)})
				\, , \]
				and lift this to
				\[
					S^\mot \coloneqq \sum_{\mathclap{\sigma \in S_{2n+1}}} \, Z^\mot(a_{\sigma(0)}, a_{\sigma(1)}, \ldots, a_{\sigma(2n)})
				\]
				on the level of motivic MZVs.  Here \( Z^\mot \) is obvious motivic version of \( Z \) given by replacing \( \zeta \) with \( \zeta^\mot \) in the definition.  The strategy is then to show a corresponding result on the motivic level first.  
				
				For each \( k \), we will show that the terms in \( D_{2k+1} S^\mot \) cancel pairwise, meaning each \( D_{2k+1} S^\mot \) is identically 0.  From \autoref{thm:kerDN} on the kernel of \( D_{<N} \) above, it follows that \( S^\mot = q \zeta^\mot(\wt) \), for some \( q \in \Q \).  Applying the period map gives this on the level of real numbers, and Euler's evaluation of \( \zeta(2k) \) shows \( S \in \pi^\wt \Q \).
			\end{proof}
			
			The proof of this proposition will proceed by a series of lemmas, the main work is in showing \( D_{2k+1} S^\mot = 0 \).  \medskip
			
			Lifting to motivic MZVs, we have by definition
			\begin{align*}
				Z^\mot(b_0, b_1, \ldots, b_{2n}) &= \zeta^\mot(\{2\}^{b_0}, 1, \{2\}^{b_1}, 3, \ldots, 1, \{2\}^{b_{2n-1}}, 3, \{2\}^{b_{2n}}) \\
					&= \pm I^\mot(0; (10)^{b_0} \, 1 \, (10)^{b_1} \, 100 \cdots 1 \, (10)^{b_{2n-1}} \, 100  \, (10)^{b_{2n}}; 1)
			\, , \end{align*}
			where the sign depends only on the depth of the MZV.  This sign is the same under any permutation of the \( b_i \), so we can safely ignore it.  Ultimately it will pull through \( D_{2k+1} \), since \( D_{2k+1} \) is linear.
			
			\begin{Def}
				The string
				\[
					0 \, (10)^{b_0} \, 1 \, (10)^{b_1} \, 100 \cdots 1 \, (10)^{b_{2n-1}} \, 100  \, (10)^{b_{2n}} \, 1
				\]
				which (after ignoring all commas and semicolons) appears as the argument of \( I^\mot \) above is the \emph{binary word} for the corresponding (motivic) MZV \( Z^\mot(b_0, b_1, \ldots, b_{2n}) \).
			\end{Def}
			
			\begin{Lem}
				The binary word for \( Z^\mot(b_0, b_1, \ldots, b_{2n}) \) can be decomposed into blocks and written more symmetrically as
				\[
					(01)^{b_0+1} \mid (10)^{b_1+1} \mid (01)^{b_2+1} \mid (10)^{b_3+1} \mid \cdots \mid (01)^{b_{2n}+1}
				\, . \]
				
				\begin{proof}
					We insert breaks, written \( \mid \), into the binary word above.  Insert a break directly after the 1 in the binary word 1 which encodes the argument 1 between \( \{2\}^{b_i} \) and \( \{2\}^{b_{i+1}} \).  Also insert a break after the 10 in the binary word 100 which encodes the argument 3 between \( \{2\}^{b_{i+1}} \) and \( \{2\}^{b_{i+2}} \).
					
					Between arguments 1 and 3 inclusive, the word looks like
					\[
						\cdots \, 1 \, (10)^{b_i} \, 100 \, \cdots
					\, . \]
					Inserting these breaks gives
					\[
						\cdots \, 1 \mid (10)^{b_i} \, 10 \mid 0 \, \cdots
					\, , \]
					and the block in the middle is \( (10)^{b_i+1} \).
					
					Between arguments 3 and 1 inclusive, the word looks like
					\[
						\cdots \, 100 \, (10)^{b_j} \, 1 \, \cdots
					\, . \]
					Inserting the breaks gives
					\[
						\cdots \, 10 \mid 0 \, (10)^{b_j} \, 1 \mid \cdots
					\, , \]
					and the middle block is \( (01)^{b_j+1} \).
					
					This pattern holds at the start of the word since
					\[
						0 \, (10)^{b_0} \, 1 \, \cdots \quad \text{becomes} \quad 0 \, (10)^{b_0} \, 1 \mid \cdots
					\, , \]
					and it holds at the end of the word since
					\[
						\cdots 100 \, (10)^{b_{2n}} \, 1 \quad \text{becomes} \quad \cdots 10 \mid 0 \, (10)^{b_{2n}} \, 1
					\, . \]
					
					Thus the entire word may be written
					\[
						(01)^{b_0+1} \mid (10)^{b_1+1} \mid (01)^{b_2+1} \mid (10)^{b_3+1} \mid \cdots \mid (01)^{b_{2n}+1}
					\]
					as claimed.
				\end{proof}
			\end{Lem}
			
			\begin{Not}
				We will identify a word of the form \( (01)^{b_0+1} \, (10)^{b_1+1} \, \cdots \, (01)^{b_{2n}+1} \) by giving the vector \( \vec{b} = [b_0, b_1, \ldots, b_{2n}] \) which determines the sizes of the blocks.  We will refer to this vector itself as the \emph{word}, and write \( I^\mot(\vec{b}) \) for the corresponding motivic iterated integral.
			\end{Not}
			
			Our goal is to compute \( D_{2k+1} S^\mot \) for each \(k \) such that \( 3 \leq 2k+1 < \wt \), and show it is identically 0.  We compute \( D_{2k+1} \) by marking out subsequences of length \( 2k+3 \) on each iterated integral in the sum \( S^\mot \).  (Remember the subsequence also includes the boundary terms.) \medskip
			
			We are going to give a more algebraic way of encoding the subsequences, so we can be sure the terms in \( D_{2k+1} S^\mot \) all cancel.
			
			\begin{Not} Encode each odd length subsequence by giving:
			\begin{itemize}
				\item the word \( \vec{b} = [b_0, b_1, \ldots, b_{2n}] \) it is taken from,
				\item the block number \( s \) it starts in (counting from 0),
				\item the number of symbols \( l \) in block \( s \) before the beginning of the subsequence,
				\item the block number \( t \) it finishes in, and
				\item the number of symbols \( m \) in block \( t \) after the end of the sequence.
			\end{itemize}
			\end{Not}
			
			For example, the subsequence
			\[
				0101 \mid 10 \underbracket{\boldsymbol{101010 \mid 01 \mid 1010101010 \mid 010}}_{\text{subsequence}} 101
			\]
			is taken from the word \( \vec{b} = [1, 3, 0, 4, 2] \).  It starts in block \( s = 1 \), with \( l = 2 \) symbols before the subsequence begins.  It finishes in block \( t = 4 \), and there are \( m = 3 \) symbols after the subsequence ends.  We encode it as \( ([1,3,0,4,2]; 1, 2; 4, 3) \). \medskip
			
			In this encoding we obviously have \( s \leq t \), as a sequence cannot finish before it starts.  We also have \( l < 2(b_s+1) \) and \( m < 2(b_t+1) \).  These conditions come from the fact that there are strictly fewer symbols before the start of a subsequence than there are symbols in the block, and similarly for the end.  In the case \( s = t \), we should also have a condition like \( l + m < 2b_s \), as the subsequence has length \( > 0 \), but for us this possibility does not arise.\medskip
			
			\begin{Def}
				If the boundary symbols (the start and end symbols) of a subsequence are the same, we will call the subsequence \emph{trivial}, because the tensor of motivic iterated integrals it corresponds to in \( D_{2k+1} \) is automatically 0.  
			\end{Def}
			
			Some facts about non-trivial odd length subsequences and their encodings:
			
			\begin{Lem}\label{lem:lmparity}
				A subsequence has odd length if and only if \( l \) and \( m \) have different parity in the encoding.
				
				\begin{proof}
					The length of the subsequence encoded as \( (\vec{b}; s, l; t, m) \) is given by \( 2(b_s+1) + 2(b_{s+1}+1) + \cdots + 2(b_t+1) - l - m \).  This is odd if and only if \( l \) and \( m \) have different parity.
				\end{proof}
			\end{Lem}
			
			\begin{Lem}
				An odd length subsequence is trivial if and only if \( s \) and \( t \) have the same parity in the encoding.
				
				\begin{proof}
					Since the number of symbols in each block is even, we may ignore any intervening blocks.  Since \( s \leq t \) we can assume \( t = s \) if they have the same parity, or \( t = s+1 \) if they have opposite parity.
					
					If \( s \) and \( t \) have the same parity, we are marking out an odd length subsequence on alternating 0s and 1s.  Such a subsequence necessarily starts and ends with the same symbol.
					
					If \( s \) and \( t \) have different parity, part way through the subsequence the pattern 01 changes to 10.  So in the latter part of the subsequence 0 and 1 have been interchanged, meaning the start and end symbols are now different.
				\end{proof}
			\end{Lem}
			
			A subsequence and its encoding can be read off from each other, so they uniquely determine each other.  Thus the non-trivial odd length subsequences on the word \( \vec{b} \) correspond bijectively to the encodings where \( s \leq t \) and they have different parity (this prevents \( s = t \), so in fact we may take \( s < t \)), where \( l \) and \( m \) have different parity, and where \( l < 2(b_s+1) \) and \( m < 2(b_t+1) \).  
			
			\begin{Def}
				We will call such an encoding above an \emph{odd encoding} of the (non-trivial) subsequence.  If the subsequence it encodes has length \( L \), we will call it an \emph{odd encoding of length \( L \)}.
			\end{Def}
			
			We are now going to define a map on these which will be used to pairwise cancel the terms of \( D_{2k+1} \).
			
			\begin{Def}\label{def:phi}
				Define the following map on odd encodings
				\[
					\phi \colon (\vec{b}; s, l; t, m) \mapsto (\vec{c}; s, m; t, l)
				\, , \]
				where \( \vec{b} = [b_0, b_1, \ldots, b_{2n}] \) and \( \vec{c} \) is defined explicitly as follows:
				\[
					c_i = \begin{cases}
						b_i & \text{if \( i < s \) or \( i > t \)} \\
						b_{s + (t - i)} & \text{if \( s \leq i \leq t \)} \\
					\end{cases}
				\]
				
				So \( \vec{c} = [b_0, b_1, \ldots, b_{s-1}, b_t, b_{t-1}, \ldots, b_{s+1}, b_{s}, b_{t+1}, \ldots, b_{2n}] \) is obtained by reversing the sequence from position \( s \) to position \( t \) inclusive in the vector \( \vec{b} \).
			\end{Def}
			
			Notice that \( \vec{c} \) is simply a permutation of \( \vec{b} \).  We can interpret this map on the binary words as reflecting the sequence of blocks \( s \) through \( t \) inclusive which contain the given subsequence.  This will produce a new subsequence on another word.
			
			\begin{Lem}\label{lem:len}
				The image of an odd encoding under \( \phi \) is again an odd encoding, moreover the length of the encoded sequence does not change.
				
				\begin{proof}
					The map does not change \( s \) or \( t \), so they are fine.  After swapping \( l, m \) to \( m, l \), they still have different parity.  Lastly we have \( m < 2(b_t+1) = 2(c_s+1)  \) and \( l < 2(b_s + 1) = 2(c_t + 1) \).
					
					The new length is given by
					\begin{align*}
						& 2(c_s + 1) + 2(c_{s+1}+1) + \cdots + 2(c_t + 1) - m - l \\
						{}={} & 2(b_t + 1) + 2(b_{t-1} + 1) + \cdots + 2(b_s + 1) - l - m
					\, , \end{align*}
					which is exactly the old length.
				\end{proof}
			\end{Lem}
			
			\begin{Lem}
				The map \( \phi \) is an involution, \( \phi^2 = \id \).
				
				\begin{proof}
					Given a odd encoding \( (\vec{b}; s, l; t, m) \), we have
					\[
						\phi^2(\vec{b}; s, l; t, m) = \phi(\vec{c}; s, m; t, l) = (\vec{d}; s, l; t, m)
					\]
					for some vector \( \vec{d} = [d_i] \).
					
					By definition, we have
					\[
						d_i = \begin{cases}
							c_i = b_i & \text{for \( i < s \) or \( i > t \)} \\
							c_{s + (t-i)} = b_{s + (t - \{s + (t-i)\})} = b_i& \text{for \( s \leq i \leq t \),}
						\end{cases}
					\]
					as we are just reversing the sequence from position \( s \) to position \( t \) a second time.  So \( \vec{d} = \vec{b} \), and \( \phi^2 = \id \).
				\end{proof}
			\end{Lem}
			
			\begin{Lem}\label{lem:sub}
				The subsequence given by an odd encoding \( \alpha \), and the subsequence given by \( \phi(\alpha) \), are the reverse of each other.
				
				\begin{proof}
					If the odd encoding is \( (\vec{b}; s, l; t, m) \), its image under \( \phi \) is \( (\vec{c}; s, m; t, l) \), where \( \vec{c} = [b_0, b_1, \ldots, b_{s-1}, b_t, b_{t-1}, \ldots, b_{s+1}, b_{s}, b_{t+1}, \ldots, b_{2n}] \) is obtained by reversing \( \vec{b} \) from position \( s \) to position \( t \).
					
					By symmetry we can assume the pattern in block \( s \) is \( 01 \), otherwise interchange 0 and 1 below.  Since \( s \) and \( t \) have different parity, the pattern in block \( t \) is \( 10 \).  We get the given by \( \alpha \) by taking the binary string 
					\[
						(01)^{b_s+1} (10)^{b_{s+1}+1} \cdots (01)^{b_{t-1}+1} (10)^{b_t+1} 
					\, , \]
					of blocks \( s \) through \( t \) inclusive of \( \vec{b} \), then removing the first \( l \) symbols and the last \( m \) symbols.  This gives the subsequence of \( \alpha \) as
					\[
						(\underbrace{\cdots 01}_{\mathclap{2(b_s+1) - l}}) (10)^{b_{s+1}+1} \cdots (01)^{b_{t-1}+1} (\underbrace{10 \cdots}_{\mathclap{2(b_t+1) - m}})
					\, . \]
					
					Correspondingly we get the subsequence given by \( \phi(\alpha) \) by taking the binary string \( (01)^{c_s+1} (10)^{c_{s+1}+1} \cdots (01)^{c_{t-1}+1} (10)^{c_{t}+1} \) of blocks \( s \) through \( t \) inclusive of \( \vec{c} \), and removing the first \( m \) symbols and last \( l \) symbols.  Recall that \( c_{i} = b_{s + (t-i)} \) for \( s \leq i \leq t \), which is given by reversing \( \vec{b} \) from position \( s \) through \( t \) inclusive.  So \( c_s = b_t \), \( c_{s+1} = b_{t-1} \), and so on.  This gives the subsequence of \( \phi(\alpha) \) as
					\begin{align*}
						& (\underbrace{\cdots 01}_{\mathclap{2(c_s+1)-m}}) (10)^{c_{s+1}+1} \cdots (01)^{c_{t-1}+1} (\underbrace{10 \cdots}_{\mathclap{2(c_{t}+1)-l}}) \\[0.75em]
						 {}={} & (\underbrace{\cdots 01}_{\mathclap{2(b_t+1)-m}}) (10)^{b_{t-1}+1} \cdots (01)^{b_{s+1}+1} (\underbrace{10 \cdots}_{\mathclap{2(b_s+1)-l}})
					\, . \end{align*}
					This is exactly the reverse of the subsequence of \( \alpha \), given above.
				\end{proof}
			\end{Lem}
			
			Recall, from Definition 4.4 in \cite[][p. 11]{DecompMot} introduced earlier, that a subsequence on a word gives rise to a quotient sequence by deleting the symbols of the word strictly between the boundary symbols of the subsequence, this is the \emph{quotient sequene given by an odd encoding}.  I now want to show how the quotient sequences given by \( \alpha \), and the qoutient sequence given by \( \phi(\alpha) \) are related. \medskip

			An explicit example first will make the abstract idea more understandable.  Consider the odd encoding \( \alpha = ([1,2,3,1,2]; 2, 2; 3, 5) \), so \( \phi(\alpha) =  ([1,3,2,1,2]; 2, 5; 3, 2) \).  The subsequences they give are
			\begin{align*}
				\alpha &\rightarrow 0101 \mid 10 \underbracket{\boldsymbol{1010 \mid 010}} 10101 \mid 1010 \mid 010101 \\
				\phi(\alpha) &\rightarrow 0101 \mid 10101 \underbracket{\boldsymbol{010 \mid 0101}} 01 \mid 1010 \mid 010101
			\, , \end{align*}
			and we can see both quotient sequences equal \( 0101 \mid 101010101 \mid 1010 \mid 010101 \).  But why is this the case?
			
			 Notice that both quotient sequences necessarily agree before block \( s = 2 \), and after block \( t = 3 \) because the words match here.  What is the contribution from blocks 2 and 3 in each case?  For \( \alpha \), the contribution from block 2 is an alternating sequence of 0's and 1's of length 3, and the contribution from block 3 is an alternating sequence of 0's and 1's of length 6.  The boundary symbols of the subsequence are different, so when we join these two contributions together we get an alternating sequence of 0's and 1's of length \( 3 + 6 = 9 \), starting with a \( 1 \).  
			 
			 But exactly the same analysis holds for \( \phi(\alpha) \).  The contribution from blocks 2 and 3 in \( \phi(\alpha) \) is an alternating sequence of 0's and 1's of length 9, starting with a \( 1 \), giving the quotient sequence above.
			
			\begin{Lem}\label{lem:quo}
				The quotient sequence given by an odd encoding \( \alpha \), and the quotient sequence given by \( \phi(\alpha) \), are equal.
			
				\begin{proof}
					Following on from the previous lemma, since \( b_i  = c_i \), for \( i < s \) and \( i > t \), the quotient sequence agree in these blocks.  Here they are both:
					\[
						(01)^{b_0+1} (10)^{b_1+1} \cdots (xy)^{b_{s-1}+1} \text{ and } (yx)^{b_{t+1}+1} \cdots (10)^{b_{2n-1}+1} (01)^{b_{2n}+1}
					\, , \]
					where \( xy \) is some pattern \( 01 \) or \( 10 \) as appropriate.  There is no contribution from blocks \( s < i < t \) as these are deleted, so we only need to consider the contribution from blocks \( s \) and \( t \) which join the two sections above.
					
					For \( \alpha \), the contribution from block \( s \) is an alternating sequence of 0's and 1's of length \( l + 1\), and the contribution from block \( t \) is an alternating sequence of 0's and 1's of length \( m + 1\).  The two boundary terms of the subsequence are different, so when we join these contributions together we get an alternating sequence of 0's and 1's of length \( l + m + 2 \).
					
					The same analysis for \( \phi(\alpha) \) shows the contribution from blocks \( s \) and \( t \) here is an alternating sequence of 0's and 1's of length \( m + l + 2 \).  These two sequences agree as they have the same length and they begin with the symbol \( y \), the first symbol in block \( s \).
					
					So both quotient sequences equal:
					\[
						(01)^{b_0+1} (10)^{b_1+1} \cdots (xy)^{b_{s-1}+1} (\underbrace{yxyx \cdots y}_{l + m + 2}) (yx)^{b_{t+1}+1} \cdots (10)^{b_{2n-1}+1} (01)^{b_{2n}+1}
					\, , \]
					and are equal as claimed.
				\end{proof}
			\end{Lem}
			
			Since \( \phi^2 = \id \), we get a group \( G = \{ \id, \phi \} \) which can act on the odd encodings:
			
			\begin{Lem}\label{lem:act}
				Let \( \mathcal{C} \) be a set of words of the form \( \vec{x} = [x_0, x_1, \ldots, x_{2n}] \), such that any permutation \( \vec{x'} = [x_{\sigma(0)}, x_{\sigma(1)}, \ldots, x_{\sigma(2n)}] \) of a word in \( \mathcal{C} \) also lies in \( \mathcal{C} \).  Then for any fixed \( L \), the group \( G \) acts on the odd encodings subsequences of length \( L \) on these words.
				
				\begin{proof}
					If an odd encoding \( \alpha \) of length \( L \) is taken from the word \( \vec{b} \), then \( \phi(\alpha) \) is an odd encoding of length \( L \) taken from the word \( \vec{c} \), where \( \vec{c} \) is a permutation of \( \vec{b} \), by \autoref{def:phi} and \autoref{lem:len}.  So we map into the set of odd encodings of length \( L \) on the words of \( \mathcal{C} \).  Function composition gives a group action on this set.
				\end{proof}
			\end{Lem}
			
			Look at the orbits of such a set under \( G \).  \emph{A priori} the orbits have size 1 or 2, the divisors of the order of \( G \).
			
			\begin{Lem}\label{lem:orbit2}
				All the orbits of odd encodings under \( G \) have size 2.
				
				\begin{proof}
					If an orbit has size 1, its unique element is fixed under \( \phi \).  This means \( l = m \), but this cannot be as they have opposite parity by \autoref{lem:lmparity}.
				\end{proof}
			\end{Lem}
			
			\begin{Lem}\label{lem:cancel}
				The two elements of a fixed orbit give terms which cancel in \( D_{2k+1} \).
				
				\begin{proof}
					The two elements are of the form \( \alpha \) and \( \phi(\alpha) \).  From \autoref{lem:sub}, they give subsequences \( X \) and \( Y \) respectively, and these are reverses of each other.  From \autoref{lem:quo}, they give the same quotient sequence \( Q \).  Hence in \( D_{2k+1} \) we get the terms \( I^\mot(X) \otimes I^\lmot(Q) \) and \( I^\lmot(Y) \otimes I^\mot(Q) \).  By reversal of paths \( I^\lmot(X) = -I^\lmot(Y) \), since the subsequence has odd length, so they cancel.
				\end{proof}
			\end{Lem}
			
			Now we can put all the pieces together and show each \( D_{2k+1} S^\mot \) is identically zero.
			
			\begin{Lem}\label{lem:Dwt}
				For each \( k \) such that \( 3 \leq 2k+1 < \wt \), we have \( D_{2k+1} S^\mot = 0 \).
			
				\begin{proof}
					The sum \( S^\mot = \pm \sum_{\sigma \in S_{2n+1}} I^\mot([a_{\sigma(0)}, a_{\sigma(1)}, \ldots, a_{\sigma(2n)}]) \) runs over all permutations in \( S_{2n+1} \), so all possible permutations of the word \( \vec{a} = [a_0, a_1, \ldots, a_{2n}] \) appear.  The sign \( \pm \) is determined by the depth of the corresponding MZVs in \( S \).
					
					In this sum each word \( \vec{a}' = [a'_0, a'_1, \ldots, a'_{2n}] \) is repeated with the same multiplicity \( \lambda \).  One can see this by counting explicitly.  The multiplicity is just the number of ways of permuting each set of repeated values of the \( a_i \)'s.  Or view \( S_{2n+1} \) as acting on these words, there is one orbit, so each stabilizer has the same size.
					
					Let \( \mathcal{C} \) be the set \( \{ [a_{\sigma(0)}, a_{\sigma(1)}, \ldots, a_{\sigma(2n)}] \mid \sigma \in S_{2n+1} \} \) of all permutations of the word \( \vec{a} \).  Then
					\[
						S^\mot = \pm \lambda \sum_{w \in \mathcal{C}} I^\mot(w)
					\, . \]
					
					Fixing  \( k \) such that \( 3 \leq 2k+1 < \wt \), we find
					\[
						D_{2k+1} S^\mot = \pm \lambda D_{2k+1} \sum_{w \in \mathcal{C}} I^\mot(w)
					\, , \]
					since \( D_{2k+1} \) is linear.  The non-zero terms of this sum are exactly the odd encodings of length \( 2k+3 \) on the words \( w \in \mathcal{C} \).  Since \( \mathcal{C} \) contains any permutations of its words, \autoref{lem:act} shows the group \( G \) acts on these encodings, and \autoref{lem:orbit2} shows they break up into orbits of size 2.  By \autoref{lem:cancel} the two elements in each orbit cancel in \( D_{2k+1} \).  Hence all terms cancel, so:
					\[
						D_{2k+1} S^\mot = \pm \lambda D_{2k+1}  \sum_{w \in \mathcal{C}} I^\mot(w) = 0
					\, , \]
					as claimed.
				\end{proof}
			\end{Lem}
			
			The rest of the proof strategy we outlined after \autoref{prop:symins} goes through without a problem:
			
			\begin{proof}[Proof of Proposition]
				We have lifted
				\[
					S \coloneqq \sum_{\mathclap{\sigma \in S_{2n+1}}} \, Z(a_{\sigma(0)}, a_{\sigma(1)}, \ldots, a_{\sigma(2n)})
				\]
				to
				\[
					S^\mot \coloneqq \sum_{\mathclap{\sigma \in S_{2n+1}}} \, Z^\mot(a_{\sigma(0)}, a_{\sigma(1)}, \ldots, a_{\sigma(2n)})
				\]
				on the motivic MZV level.  By \autoref{lem:Dwt}, \( D_{<\wt} S^\mot = 0 \), so \autoref{thm:kerDN} on the kernel of \( D_{<N} \) tells us that \( S^\mot = q\zeta^\mot(\wt) \), for some \( q \in \Q \).
				
				Apply the period map in \autoref{eqn:permap} to this, and we get
				\[
					S = \per S^\mot = \per q \zeta^\mot(\wt) = q \zeta(\wt)
				\, . \]
				The weight of each MZV in the sum is even, explicitly it is \( \wt = 4n + 2 \sum_{i=0}^{2n} a_i \).  Euler's evaluation of \( \zeta(2k) \) says
				\[ 
					\zeta(2k) = (-1)^{k+1} \frac{B_{2n} (2\pi)^{2n}}{2 (2n)!}
				\, , \]
				where \( B_{2n} \) is a Bernoulli number and in particular rational.  This shows that \( \zeta(\wt) \in \pi^\wt \Q \).  Hence \( S = q \zeta(\wt) \in \pi^\wt \Q \), as claimed.
			\end{proof}
			
			As a corollary to this we have:
					
			\begin{Cor}
				For given integers \( m, n \geq 0 \), the MZV
				\[ 
					\zeta(\{ \, \{2\}^m, 1, \{2\}^m, 3\}^n, \{2\}^m)
				\]
				is a rational multiple of \( \pi^{4n + 2m(2n+1)} \).
	
				\begin{proof}
					Put \( B \coloneqq \zeta(\{ \, \{2\}^m, 1, \{2\}^m, 3\}^n, \{2\}^m) \).  Now set \( a_0 = a_1 = \cdots = a_{2n} = m \) in the above result.  For any permutation \( \sigma \in S_{2n+1} \), we have:
					\[
				 Z(a_{\sigma(0)}, a_{\sigma(1)}, \ldots, a_{\sigma(2n)}) = Z(m, m, \ldots, m) = B
					\]
					Summing over all permutations we get:
					\[
						(2n+1)! \, B = \sum_{\mathclap{\sigma \in S_{2n+1}}} \, Z(a_{\sigma(0)}, a_{\sigma(1)}, \ldots, a_{\sigma(2n)}) \in \pi^\wt \Q
					\]
					by symmetric insertion.  Dividing by \( (2n+1)! \) shows \( B \in \pi^\wt \Q \).
					
					In this case the weight of \( B \) is \( \wt = 4n + 2 \sum_{i=0}^{2n} m = 4n + 2m(2n+1) \), so \( B \) is a rational multiple of \( \pi^{4n + 2m(2n+1)} \) as claimed.
				\end{proof}
			\end{Cor}
		
	\printbibliography

@article{EvalEZ,
	author = "J. M. Borwein and D. M. Bradley and D. J. Broadhurst",
	title = "Evaluations of $k$-fold Euler/Zagier sums: a compendium of results for arbitrary $k$",
	journal = "The Electronic Journal of Combinatorics",
	year = "1997",
	volume = "4",
	number = "2",
	pages = "\#R5",
	eprinttype = "arxiv",
	eprint = "hep-th/9611004"
}

@incollection{DecompMot,
	booktitle = {Galois-Teichm\"{u}ller Theory and Arithmetic Geometry},
	series = "Advanced Studies in Pure Mathematics",
	volume = "63",
	year = "2012",
	publisher = "The Mathematical Society of Japan",
	author = "Francis C. S. Brown",
	title = "On the decomposition of motivic multiple zeta values",
	pages = "31--58",
	eprinttype = "arxiv",
	eprint = "1102.1310",
	eprintclass = "math.NT"
}

@article{BrownMTM,
	title = "Mixed Tate Motives over $\Z$",
	author = "Francis Brown",
	journal = "Annals of Mathematics",
	volume = "175",
	number = "2",
	pages = "949-976",
	year = "2012",
	eprinttype = "arxiv",
	eprint = "1102.1312",
	eprintclass = "math.AG"
}

@article{CombAspectsMZV,
	author = "Jonathan M. Borwein and David M. Bradley and David J. Broadhurst and Petr Lison\v{e}k",
	title = "Combinatorial aspects of multiple zeta values",
	journal = "The Electronic Journal of Combinatorics",
	year = "1998",
	volume = "5",
	pages = "\#R38",
	eprinttype = "arxiv",
	eprint = "math/9812020",
	eprintclass = "math.NT"
}

@article{AlgCombMZV,
	author = "Douglas Bowman and David M. Bradley",
	title = "The Algebra and Combinatorics of Shuffles and Multiple Zeta Values",
	journal = "Journal of Combinatorial Theory, Series A",
	year = "2002",
	volume = "97",
	number = "1",
	pages = "43--61",
	eprinttype = "arxiv",
	eprint = "math/0310082",
	eprintclass = "math.CO"
}

@article{MunetaEvalMZV,
	author = "Shuichi Muneta",
	title = "A note on evaluations of multiple zeta values",
	journal = "Proceedings of the American Mathematical Society",
	year = "2009",
	volume = "137",
	number = "3",
	pages = "931--935",
	eprinttype = "arxiv",
	eprint = "0802.4331",
	eprintclass = "math.NT"
}

@article{ZhaoExoticShuffle,
	author = "Jianqiang Zhao",
	title = "An exotic shuffle relation for multiple zeta values",
	journal = "Archiv der Mathematik",
	year = "2008",
	volume = "91",
	number = "5",
	pages = "409-415",
	eprinttype = "arxiv",
	eprint = "0707.3244",
	eprintclass = "math.NT"
}

@article{GaloisSymGon,
	title = "Galois symmetries of fundamental groupoids and noncommutative geometry",
	author = "A. B. Goncharov",
	journal = "Duke Mathematical Journal",
	volume = "128",
	number = "2",
	year = "2005",
	pages = "209--284",
	eprinttype = "arxiv",
	eprint = "math/0208144",
	eprintclass = "math.AG"
}

@inproceedings{MZVApp,
	title = "Values of Zeta Functions and their applications",
	author = "D. Zagier",
	booktitle = "Proceedings of the ECM Paris 1992",
	year = "1994",
	pages = "497--512",
	series = "Progress in Mathematics",
	volume = "120"
}

\end{document}